\documentclass{amsart}[12pt]
\usepackage{amsmath,amssymb,amsthm,amscd}

\newtheorem{theorem}{Theorem}
\newtheorem{lemma}[theorem]{Lemma}
\newtheorem{corollary}[theorem]{Corollary}

\newtheorem*{ack}{Acknowledgements}
\newcommand{\Z}{\mathbb{Z}}

\newcommand{\F}{\mathbb{F}}
\newcommand{\Q}{\mathbb{Q}}
\newcommand{\C}{\mathbb{C}}
\newcommand{\la}{\langle}
\newcommand{\ra}{\rangle}
\newcommand{\Gal}{{\rm Gal}}
\newcommand{\AGL}{{\rm AGL}}
\newcommand{\GL}{{\rm GL}}
\newcommand{\ord}{{\rm ord}}
\newcommand{\artin}[2]{\genfrac{[}{]}{}{}{#1}{#2}}
\begin{document}

\title[The Somos-5 sequence]{The density of primes dividing a term in the Somos-5 sequence}
\author{Bryant Davis}
\author{Rebecca Kotsonis}
\author{Jeremy Rouse}
\address{Department of Mathematics, Wake Forest University, Winston-Salem, NC 27109}
\email{davibf11@wfu.edu}
\email{kotsrc11@wfu.edu}
\email{rouseja@wfu.edu}

\begin{abstract}
The Somos-5 sequence is defined by $a_{0} = a_{1} = a_{2} = a_{3} = a_{4} = 1$
and $a_{m} = \frac{a_{m-1} a_{m-4} + a_{m-2} a_{m-3}}{a_{m-5}}$ for $m \geq 5$. We
relate the arithmetic of the Somos-5 sequence to the elliptic curve $E : y^{2} + xy = x^{3} + x^{2} - 2x$ and use properties of Galois representations attached to $E$ to prove the density of primes $p$ dividing
some term in the Somos-5 sequence is equal to $\frac{5087}{10752}$.
\end{abstract}

\maketitle

\section{Introduction and Statement of Results}

There are many results in number theory that relate to a determination of
the primes dividing some particular sequence. For example, it well-known
that if $p$ is a prime number, then $p$ divides some term of the Fibonacci
sequence, defined by $F_{0} = 0$, $F_{1} = 1$, and $F_{n} = F_{n-1} + F_{n-2}$
for $n \geq 2$. Students in elementary number theory learn
that a prime $p$ divides a number of the form $n^{2} + 1$ if and only if
$p = 2$ or $p \equiv 1 \pmod{4}$.

In 1966, Hasse proved in \cite{Hasse} that if $\pi_{{\rm even}}(x)$ is the
number of primes $p \leq x$ so that $p | 2^{n} + 1$ for some $n$, then
\[
  \lim_{x \to \infty} \frac{\pi_{{\rm even}}(x)}{\pi(x)} = \frac{17}{24}.
\]
Note that a prime number $p$ divides $2^{n} + 1$ if and only if
$2$ has even order in $\F_{p}^{\times}$.

A related result is the following. The Lucas numbers are defined by
$L_{0} = 2$, $L_{1} = 1$ and $L_{n} = L_{n-1} + L_{n-2}$ for $n \geq
2$. In 1985, Lagarias proved (see \cite{Lagarias} and
\cite{Lagarias2}) that the density of primes dividing some Lucas number
is $2/3$. Given a prime number $p$, let $Z(p)$ be the smallest integer $m$
so that $p | F_{m}$. A prime $p$ divides $L_{n}$ for some $n$ if and only if
$Z(p)$ is even. In \cite{CubreRouse}, Paul Cubre and the third author
prove a conjecture of Bruckman and Anderson on the density of primes $p$
for which $m | Z(p)$, for an arbitrary positive integer $m$.

In the early 1980s, Michael Somos discovered integer-valued 
non-linear recurrence sequences. The Somos-$k$ sequence is defined by
$c_{0} = c_{1} = \cdots = c_{k-1} = 1$ and
\[
  c_{m} = \frac{c_{m-1} c_{m-(k-1)} + c_{m-2} c_{m-(k-2)} + \cdots
  + c_{m-\lfloor \frac{k}{2} \rfloor} c_{m-\lceil \frac{k}{2} \rceil}}{c_{m-k}}
\]
for $m \geq k$. Despite the fact that division is involved in the definition
of the Somos sequences, the values $c_{m}$ are integral for $4 \leq k \leq 7$.
Fomin and Zelevinsky \cite{FZ} show that the introduction of parameters
into the recurrence results in the $c_{m}$ being Laurent polynomials
in those parameters. Also, Speyer \cite{SP} gave a combinatorial interpretation
of the Somos sequences in terms of the number of perfect matchings in
a family of graphs. 

Somos-4 and Somos-5 type sequences are also connected with the
arithmetic of elliptic curves (a connection made quite explicit by
A. N. W. Hone in \cite{Hone1}, and \cite{Hone2}). If $a_{n}$ is the
$n$th term in the Somos-4 sequence, $E : y^{2} + y = x^{3} - x$ and
$P = (0,0) \in E(\Q)$, then the $x$-coordinate of the denominator of
$(2n-3)P$ is equal to $a_{n}^{2}$.  It follows from this that
$p | a_{n}$ if and only if $(2n-3)P$ reduces to the identity in
$E(\F_{p})$, and so a prime $p$ divides a term in the Somos-4 sequence
if and only if $(0,0) \in E(\F_{p})$ has odd order. In
\cite{JonesRouse}, Rafe Jones and the third author prove that the
density of primes dividing some term of the Somos-4 sequence is
$\frac{11}{21}$. The goal of the present paper is to prove an
analogous result for the Somos-5 sequence.

Let $\pi'(x)$ denote the number of primes $p \leq x$ so that $p$ divides
some term in the Somos-5 sequence. We have the following table of data.

\begin{center}
\begin{tabular}{c|cc}
$x$ & $\pi'(x)$ & $\frac{\pi'(x)}{\pi(x)}$\\
\hline
$10$ & $3$ & $0.750000$\\
$10^{2}$ & $12$ & $0.480000$\\
$10^{3}$ & $83$ & $0.494048$\\
$10^{4}$ & $588$ & $0.478438$\\
$10^{5}$ & $4539$ & $0.473207$\\
$10^{6}$ & $37075$ & $0.472305$\\
$10^{7}$ & $314485$ & $0.473209$\\
$10^{8}$ & $2725670$ & $0.473087$\\
$10^{9}$ & $24057711$ & $0.473134$\\
$10^{10}$ & $215298607$ & $0.473129$
\end{tabular}
\end{center}

Our main result is the following.
\begin{theorem}
\label{main}
We have
\[
  \lim_{x \to \infty} \frac{\pi'(x)}{\pi(x)} = \frac{5087}{10752}
  \approx 0.473121.
\]
\end{theorem}

The Somos-5 sequence is related to the coordinates of rational points
on the elliptic curve $E : y^{2} + xy = x^{3} + x^{2} - 2x$. This curve has $E(\Q) \cong \Z \times \Z/2\Z$ and generators are $P = (2,2)$ (of infinite order) and $Q = (0,0)$ (of order $2$). We have (see Lemma~\ref{somos5form}) that
\[
  mP + Q = \left(\frac{a_{m+2}^{2} - a_{m} a_{m+4}}{a_{m+2}^{2}},
  \frac{4 a_{m} a_{m+2} a_{m+4} - a_{m}^{2} a_{m+6} - a_{m+2}^{3}}{a_{m+2}^{3}}\right).
\]
It follows that a prime $p$ divides a term in the Somos-5 sequence if and only
if the reduction of $Q$ modulo $p$ is in
$\langle P \rangle \subseteq E(\F_{p})$. Another way of stating this is the following: there is a $2$-isogeny $\phi : E \to E'$, where $E' : y^{2} + xy =
x^{3} + x^{2} + 8x + 10$ and
\[
  \phi(x,y) = \left(\frac{x^{2} - 2}{x}, \frac{x^{2} y + 2x + 2y}{x^{2}}\right).
\]
The kernel of $\phi$ is $\{ 0, Q \}$. Letting $R = \phi(P)$ we show
(see Theorem~\ref{somosorder}) that $p$ divides some term in the Somos-5 sequence
if and only if the order of $P$ in $E(\F_{p})$ is twice that of
$R$ in $E'(\F_{p})$.

A result of Pink (see Proposition 3.2 on page 284 of \cite{Pink})
shows that the $\ell$-adic valuation of the order of a point $P
\pmod{p}$ can be determined from a suitable Galois representation
attached to an elliptic curve. For a positive integer $k$, we let
$K_{k}$ be the field obtained by adjoining to $\Q$ the $x$ and $y$
coordinates of all points $\beta_{k}$ with $2^{k} \beta_{k} = P$. There is
a Galois representation $\rho_{E,2^{k}} : \Gal(K_{k}/\Q) \to \AGL_{2}(\Z/2^{k} \Z)$
and we relate the power of $2$ dividing the order of $P$ in $E(\F_{p})$
to $\rho_{E,2^{k}}(\sigma_{p})$, where $\sigma_{p}$ is a Frobenius
automorphism at $p$ in $\Gal(K_{k}/\Q)$. Using the isogeny $\phi$ we are able
to relate $\rho_{E,2^{k}}(\sigma_{p})$ and $\rho_{E',2^{k-1}}(\sigma_{p})$,
obtaining a criterion that indicates when $p$ divides some term in the Somos-5
sequence. We then determine the image of $\rho_{E,2^{k}}$ for all $k$.

Once the image of $\rho_{E,2^{k}}$ is known, the problem of computing
the fraction of elements in the image with the desired properties is
quite a difficult one. We introduce a new and simple method for computing
this fraction and apply it to prove Theorem~\ref{main}.

\begin{ack}
The first and second authors thank the Wake Forest
Undergraduate Research and Creative Activities Center for financial support.
The authors used Magma \cite{Magma} version 2.20-6 for computations.
\end{ack}

\section{Background}

If $E/F$ is an elliptic curve given in the form
$y^{2} + a_{1} xy + a_{3} y = x^{3} + a_{2} x^{2} + a_{4} x + a_{6}$,
the set $E(F)$ has the structure of an abelian group. Specifically, if
$P, Q \in E(F)$, let $R = (x,y)$ be the third point of intersection
between $E$ and the line through $P$ and $Q$. We define
$P+Q = (x,-y-a_{1} x - a_{3})$.  The multiplication by $m$ map on an
elliptic curve has degree $m^{2}$, and so if $E/\C$ is an elliptic
curve and $\alpha \in E(\C)$, then there are $m^{2}$ points $\beta$ so
that $m \beta = \alpha$. 

If $K/\mathbb{Q}$ is a finite extension, let $\mathcal{O}_K$ denote the ring of algebraic integers in $K$. A prime $p$ ramifies in $K$ if $p\mathcal{O}_K = \prod_{i=1}^{r} \mathfrak{p}_i^{e_i}$ and some $e_i > 1$, where the $\mathfrak{p}_i$ are distinct prime ideals of $\mathcal{O}_K$.

Suppose $K/\Q$ is Galois, $p$ is a prime number that does not ramify
in $K$, and $p\mathcal{O}_K = \prod_{i=1}^{g} \mathfrak{p}_i$. For each
$i$, there is a unique element $\sigma \in \Gal(K/\Q)$ for which
\[
  \sigma(\alpha) \equiv \alpha^{p} \pmod{\mathfrak{p}_{i}}
\]
for all $\alpha \in \mathcal{O}_{K}$. This element is called the
Artin symbol of $\mathfrak{p}_{i}$ and is denoted 
$\artin{K/\Q}{\mathfrak{p}_{i}}$. If $i \neq j$,
$\left[ \frac{K/\mathbb{Q}}{\mathfrak{p}_i} \right]$ and
$\left[ \frac{K/\mathbb{Q}}{\mathfrak{p}_j} \right]$ are conjugate in
$\Gal(K/\Q)$ and
$\artin{K/\Q}{p} := \left\{ \left[
    \frac{K/\mathbb{Q}}{\mathfrak{p}_i} \right] : 1 \leq i \leq g
\right\}$ is a conjugacy class in $\Gal(K/\mathbb{Q})$.

The key tool we will use in the proof of Theorem~\ref{main} is the Chebotarev
density theorem.
\begin{theorem}[\cite{IK}, page 143]
If $C \subseteq \Gal(K/\Q)$ is a conjugacy class, then
\[
  \lim_{x \to \infty}
  \frac{\# \{ p \leq x : p \text{ prime}, \artin{K/\Q}{p} = C \}}{\pi(x)}
  = \frac{|C|}{|\Gal(K/\Q)|}.
\]
\end{theorem}
Roughly speaking, each element of $\Gal(K/\Q)$ arises as $\left[ \frac{K/\Q}{\mathfrak{p}} \right]$ equally often. 

Let $E[m] = \{ P \in E : mP = 0 \}$ be the set of points of order dividing
$m$ on $E$. Then $\Q(E[m])/\Q$ is Galois and $\Gal(\Q(E[m])/\Q)$ is isomorphic
to a subgroup of ${\rm Aut}(E[m]) \cong \GL_{2}(\Z/m\Z)$. Moreover,
Proposition V.2.3 of \cite{Silverman} implies that if $\sigma_p$ is a Frobenius automorphism at some prime above $p$ and $\tau : \Gal(\Q(E[m])/\Q) \rightarrow \GL_{2}(\Z/m \Z)$, then ${\rm tr}~\tau(\sigma_{p}) \equiv p+1 - \# E(\F_{p}) \pmod{m}$ and $\det(\tau(\sigma_{p})) \equiv p \pmod{m}$. Another
useful fact is the following. If $K/\Q$ is a number field, $\mathfrak{p}$
is a prime ideal in $\mathcal{O}_{K}$ above $p$, $\gcd(m,p) = 1$ and $P \in E(K)[m]$ is not the identity, then $P$ does not reduce to the identity in
$E(\mathcal{O}_{K}/\mathfrak{p})$. This is a consequence of
Proposition VII.3.1 of \cite{Silverman}.

We will construct Galois representations attached to elliptic curves
with images in $\AGL_{2}(\Z/2^{k} \Z) \cong (\Z/2^{k} \Z)^{2} \rtimes
\GL_{2}(\Z/2^{k} \Z)$. Elements of such a group can be thought of
either as pairs $(\vec{v}, M)$, where $\vec{v}$ is a row vector,
and $M \in \GL_{2}(\Z/2^{k} \Z)$, or as $3 \times 3$ matrices
$\begin{bmatrix} a & b & 0 \\ c & d & 0 \\ e & f & 1 \end{bmatrix}$,
where $\vec{v} = \begin{bmatrix} e & f \end{bmatrix}$ and $M
= \begin{bmatrix} a & b \\ c & d \end{bmatrix}$. In the former notation,
the group operation is given by
\[
(\vec v_1, M_1)*(\vec v_2, M_2) = (\vec v_1 + \vec v_2M_1, M_2M_1).
\]

\section{Connection between the Somos-5 sequence and $E$}

\begin{lemma}
\label{somos5form}
Define $P = (2,2)$ and $Q = (0,0)$ on $E : y^{2} + xy = x^{3} + x^2 - 2x$. For all $m \geq 0$, we have the following relationship between the Somos-5 sequence and $E$:
\begin{align*}
mP + Q = \left(\frac{a_{m+2}^2 - a_{m}a_{m+4}}{a_{m+2}^2} , \frac{4a_{m}a_{m+2}a_{m+4} - a_{m}^2 a_{m+6} - a_{m+2}^3}{a_{m+2}^3}\right).
\end{align*}
\end{lemma}

\begin{proof} We will prove this by strong induction. A straightforward calculation shows that the base cases $m = 0$ and $m = 1$ are true. For simplicity's sake, we will denote $a = a_{m}$, $b = a_{m+1}$, $c = a_{m+2}$, $d = a_{m+3}$, $e = a_{m+4}$, $f = a_{m+5}$ and $g = a_{m+6}$. Our inductive hypothesis is that
\begin{align*}
mP + Q = \left(\frac{c^2 - ae}{c^2} , \frac{4ace - a^2g - c^3}{c^3}\right).
\end{align*}
We will now compute $(m+2)P + Q$. 

To find the $x$ and $y$ coordinates of $(m+2)P + Q$, we add $2P = (1,-1)$
to $mP+Q$. If $w$ is the slope and $v$ is the $y$-intercept, the line between $2P$ and $mP + Q$ is $y = wx + v$ with $w = \frac{ag - 4ce}{ce}$ and $v = \frac{-ag + 3ce}{ce}$. Substituting this into the equation for $E$, we find the $x$-coordinate of $2P + (mP + Q)$ to be $r_x = \frac{a^2g^2 - 7aceg + ae^3 + 9c^2e^2}{c^2e^2}$. A straightforward but lengthy inductive calculation shows that if
\begin{align*}
F(a,c,e,g) = a^2g^2 - 7aceg + ae^3 + c^3g + 8c^2e^2,
\end{align*}
then $F(a_n, a_{n+2}, a_{n+4}, a_{n+6}) = 0$ for all $n$. Since $F(a,c,e,g)=0$, we know that $r_x - F(a,c,e,g) = r_x$. Therefore, we know that $r_x = \frac{-cg + e^2}{e^2}$.

Denote the $y$-coordinate of $(m+2)P + Q$ as $r_y$. We compute that
$r_y = \frac{g(ag - 3ce)}{e^3}$. Using that
$r_{y} = r_{y} - \frac{F(a,c,e,g)}{ae^{3}}$, we find that
$r_y = \frac{4ceg - c^2i - e^3}{e^3}$. Therefore, it is evident that
$$
(m+2)P + Q = \left(\frac{a_{m+4}^2 - a_{m+2}a_{m+6}}{a_{m+4}^2}, \frac{4a_{m+2}a_{m+4}a_{m+6} - a_{m+2}^2a_{m+8} -a_{m+4}^3}{a_{m+4}^3}\right).
$$
\end{proof}

Let $E'$ be given by $E' : y^2 + xy = x^3 + x^2 + 8x + 10$
and let $R = (1,4) \in E'(\Q)$. We have a $2$-isogeny $\phi : E \to E'$
given by
\[
  \phi(x,y) = \left(\frac{x^{2} - 2}{x}, \frac{x^{2} y + 2x + 2y}{x^{2}}\right).
\]

\begin{theorem}
\label{somosorder}
If $p$ is a prime that divides a term in the Somos-5 sequence, the order of $P = (2,2)$ in $E(\mathbb{F}_p)$ is twice the order of $R=(1,4)$ in $E'(\mathbb{F}_p)$. Otherwise, their order is the same.
\end{theorem}

\begin{proof}
  If $p$ divides a term in our sequence, say $a_m$, we know from our
  previous lemma that the denominators $(m-2)P + Q$ are divisible by
  $p$. Therefore, modulo $p$, $(m-2)P + Q = 0$. The point $Q$ has
  order 2, so adding $Q$ to both sides we know that $(m-2)P =
  Q$.
  Therefore, we can deduce that $Q \in \la P \ra$. We have
  $\text{ker}(\phi) = \{Q, 0 \}$ (see Section 3.4 of
  \cite{SilvermanAndTate}). Therefore, if $\phi$ is restricted to the
  subgroup generated by $P$, we have $|\text{ker}(\phi)| = 2$.  Since
  $\phi(P) = R$, by the first isomorphism theorem for groups,
  $\frac{|\la P \ra|}{|{\rm ker}(\phi)|} = |\la R \ra|$. It follows that
  $|P| = 2 \cdot |R|$.

  Alternatively, assume $p$ does not divide a term in the Somos-5
  sequence. So, there is no $m$ such that $mP + Q = 0$ modulo $p$,
  which implies that $Q \not \in \la P \ra$. Therefore, the kernel
  of $\phi$ restricted to $\la P \ra$ is $\{ 0 \}$ and so
  $|P| = |\phi(P)| = |R|$.
\end{proof}

\section{Galois representations}

Denote by $E[2^r]$ the set of points on $E$ with order dividing
$2^r$. Denote $K_r$ as the field obtained by adjoining to $\Q$ all $x$
and $y$ coordinates of points $\beta$ with $2^{r} \beta = P$. For a
prime $p$ that is unramified in $K_r$, let
$\sigma = \left[ \frac{K_{r}/\Q}{\mathfrak{p}_{i}} \right]$ for some
prime ideal $\mathfrak{p}_{i}$ above $p$. Given a basis
$\langle A,B \rangle$ for $E[2^r]$, for any such
$\sigma \in \Gal(K_{r}/\Q)$, we have
$\sigma(\beta) = \beta + e A + f B$. Also, $\sigma(A) = a A + b B$ and
$\sigma(B) = c A + d B$. Define the map
$\rho_{E, 2^k}: \Gal(K_{r}/ \Q) \rightarrow \AGL_2(\Z / 2^k \Z)$ by
$\rho_{E,2^k}(\sigma) = (\vec v, M)$ where
$M = \bmatrix a & b \\ c & d \endbmatrix$ and
$\vec v = \bmatrix e & f \endbmatrix$. Let
$\tau : \Gal(K_r/\Q) \rightarrow \GL_2(\Z/2^k\Z)$ be given by
$\tau(\sigma) = M$. 

Let $S = \left\{ \beta \in E(\C) : m\cdot \beta \in E(K)\right\}$ and
let $L$ be the field obtained by adjoining all $x$ and $y$ coordinates
of points in $S$ to $K$. Then the only primes $p$ that ramify in $L/K$
are those that divide $m$ and those where $E/K$ has bad reduction (see
Proposition VIII.1.5(b) in \cite{Silverman}). For
$E : y^{2} + xy = x^{3} + x^{2} - 2x$, the conductor of $E$ is
$102 = 2 \cdot 3 \cdot 17$ and so the only primes that ramify in
$K_{r}/\Q$ are $2$, $3$ and $17$.

Note that, if $p$ is unramified, there are multiple primes
$\mathfrak{p}_i$ above $p$ which could result in different matrices
$M_i$ and $\vec v_i$. However, properties we consider of these
$\vec v_i$ and $M_i$ do not depend on the specific choice of
$\mathfrak{p}_i$. The map depends on the choice of basis for $E[2^{r}]$,
we choose this basis as described below in Theorem~\ref{choosebasis}.

Let $\beta_{r} \in E(\C)$ be a point with $2^r\beta_r = P$. We say
that $\beta_r$ is an $r^{\text{th}}$ preimage of $P$ under
multiplication by 2. Let $p$ be a prime with $p \ne 2$, $3$ or $17$,
$\sigma = \artin{K_{r}/\Q}{\mathfrak{p}_{i}}$, and
$(\vec{v},M) = \rho_{E,2^{r}}(\sigma)$. Assume that
$\det(I - M) \not\equiv 0 \pmod{2^r}$.  This implies that
$\# E(\F_{p}) \not\equiv 0 \pmod{2^{r}}$.

\begin{theorem}
Assume the notation above. Then $2^hP$ has odd order in $E(\F_{p})$ if and only if $2^h \vec v$ is in the image of $I-M$.
\end{theorem}
\begin{proof}
  First, assume $2^h \vec v$ is in the image of $I-M$. This means that
  $\vec x = 2^h\vec v + \vec xM$ for
  some row vector $\vec x$ with coordinates in
  $(\mathbb{Z} / 2^r \mathbb{Z})^2$. If this is true for
  $\vec x = \begin{bmatrix} e & f \end{bmatrix}$, define $C := 2^{h} \beta_{r} + eA + fB$. We know then that $\sigma(C) = C$. We have $2^{r} C
= 2^{h} (2^{r} \beta_{r}) = 2^{h} P$. If $|C|$ is odd, then clearly $|2^rC| = |2^hP|$ is also odd. 

If $|C|$ is even, then every multiplication of $C$ by 2 cuts the order by a factor of 2 until we arrive at a point of odd order. Since $|E(\F_{p})|
\equiv \det(I-M) \not\equiv 0 \pmod{2^{r}}$, the power of 2 dividing $|C|$ is also less than $r$, and so $|2^rC| = |2^hP|$ is odd.

Conversely, assume that $|2^hP|$ is odd. Let $a$ be the multiplicative
inverse of $2^{r}$ modulo $|2^{h} P|$ and define $C := a2^{h} P \in E(\F_{p})$.
Then $2^{r} S = 2^{h} P$ and so $2^{r} (C - 2^{h} \beta_{r}) = 0$. It follows
that $C = 2^{h} \beta_{r} + yA + zB \in E(\F_{p})$ for some $y, z \in \Z/2^{r} \Z$. This implies that there is a Frobenius automorphism $\sigma \in \Gal(K_{r}/\Q)$ for which $\sigma(C) \equiv C \pmod{\mathfrak{p}_{i}}$ for some prime
ideal $\mathfrak{p}_{i}$ above $p$.

We claim that $\sigma(C) = C$ (as elements of $E(K_{r})$). Note that
$\sigma(C) - C \in E[2^{r}]$ and $\sigma(C) - C$ reduces to the
identity modulo $\mathfrak{p}_{i}$. Since reduction is injective on
torsion points of order coprime to the characteristic, and $p$ is odd,
it follows that $\sigma(C) = C$. It follows that if $\rho_{E,2^{r}}(\sigma)
= (\vec{v},M)$ then $2^{h} \vec{v} = (I-M) \begin{bmatrix} y & z \end{bmatrix}$,
which implies that $2^{h} \vec{v}$ is in the image of $I-M$.
\end{proof}

The following corollary is immediate. 
\begin{corollary}
Let $o$ be the smallest positive integer so that $2^o\vec v = (I-M)$$\vec x$ for some $\vec x$ with entries in $(\mathbb{Z} / 2^r \mathbb{Z})^2$. Then $2^o$ is the highest power of 2 dividing $|P|$.
\end{corollary}

The following theorem gives a convenient choice of basis for
$E[2^{k}]$ and $E'[2^{k}]$.
\begin{theorem}
\label{choosebasis}
Given a positive integer $k$, there are points $A_k, B_k \in E(\mathbb{C})$ that generate $E[2^k]$ and points $C_k, D_k \in E'(\mathbb{C})$ that generate $E'[2^k]$ so that $\phi(A_k) = C_k$ and $\phi(B_k) = 2D_k$. These points also satisfy the relations:
\[
2A_k = A_{k-1}, \quad
2B_k = B_{k-1}, \quad
2C_k = C_{k-1}, \text{ and} \quad
2D_k = D_{k-1}.
\]
\end{theorem}
\begin{proof}
We will prove this by induction. Recall that $\phi : E \rightarrow E'$ is the 
isogeny with ker $\phi = \{0,T\}$ where $T = (0,0)$. Let $\phi'
: E' \to E$ be the dual isogeny, and note that $\phi \circ \phi'(P) = 2P$.
\textit {Base Case}: Let $ k = 1$. We want to find $\langle A_1, B_1 \rangle$ 
to generate $E[2]$ and $\langle C_1, D_1 \rangle$ to generate $E'[2]$ so that 
$\phi(A_1) = C_1$ and $\phi(B_1) = 2D_1$. We set $B_{1} = (0,0)$, and choose 
$A_{1}$ to be any non-identity point in $E[2]$ other than $(0,0)$. We set 
$C_{1} = \phi(A_{1}) = (-5/4,5/8)$ and choose $D_{1}$ to be any non-identity 
point in $E'[2]$ other than $C_{1}$. Note that $\phi'(D_{1}) = B_{1}$. 

\textit {Inductive Hypothesis}: Assume $\langle A_k, B_k \rangle = E[2^k]$ and $ \langle C_k, D_k \rangle = E'[2^k]$ so that $\phi(A_k) = C_k$, $\phi(B_k) = 2D_k$, and $\phi'(D_{k}) = B_{k}$. Moreover, $D_{k} \not\in \phi(E[2^{k}])$.

Since $|\ker \phi| = 2$, we have that $\phi(E[2^{k+1}]) \supseteq E'[2^{k}]$.
Hence, we can choose $B_{k+1}$ so that $\phi(B_{k+1}) = D_{k}$.
Then $2B_{k+1} = \phi'(\phi(B_{k+1})) = \phi'(D_{k}) = B_{k}$. We choose
$D_{k+1}$ so that $\phi'(D_{k+1}) = B_{k+1}$. Note that
$2D_{k+1} = \phi(B_{k+1}) = D_{k}$ and so $D_{k+1} \in E'[2^{k+1}]$.
Now we pick $A_{k+1}$ so that $2A_{k+1} = A_{k}$ and define $C_{k+1} = \phi(A_{k+1})$. 

By our Inductive Hypothesis, $\langle A_k, B_k \rangle = E[2^k]$. This implies that $\langle A_k \rangle \cap \langle B_k \rangle = 0$, which in turn implies that $\langle 2A_{k+1}\rangle \cap \langle 2B_{k+1} \rangle = 0$. Let $C \in \langle A_{k+1} \rangle \cap \langle B_{k+1} \rangle$. Then, $C = aA_{k+1} = bB_{k+1}$. Because $|g^m| = \frac{|g|}{\gcd(m,|g|)}$, $|c| = \frac{2^{k+1}}{2^{\ord_2(a)}} = \frac{2^{k+1}}{2^{\ord_2(b)}}$, where $ \ord_2(n) $ is the highest power of $2$ dividing $n$, it follows that either $a$ and $b$ are both even, or they are both odd.
If $a$ and $b$ are even, then $C \in \langle A_k \rangle \cap \langle B_k \rangle = 0$, which is a contradiction. If $a$ and $b$ are odd, then
$|C| = 2^{k+1}$ but $2C \in \langle A_k \rangle \cap \langle B_k \rangle = 0$, which is also a contradiction. It follows
that $\langle A_{k+1} \rangle \cap \langle B_{k+1} \rangle = 0$,
which gives that $E[2^{k+1}] = \langle A_{k+1}, B_{k+1} \rangle$.

Now we show that $\langle C_{k+1}, D_{k+1}\rangle = E'[2^{k+1}]$, by
way of showing that $\langle C_{k+1}\rangle \cap \langle D_{k+1}\rangle = 0$. We 
have shown that $\langle A_{k+1}, B_{k+1} \rangle = E[2^{k+1}]$, and so
$\phi(E[2^{k+1}]) = \langle C_{k+1}, 2D_{k+1} \rangle$.  We want to
show that $D_{k+1} \notin \phi(E[2^{k+1}])$.

If $D_{k+1} \in \phi(E[2^{k+1}])$, then $D_{k+1} = aC_{k+1} + 2bD_{k+1}$. So, $aC_{k+1} + (2b - 1)D_{k+1} = 0$. Since $(2b - 1)$ is odd, 
$(2b - 1)D_{k+1}$ has order dividing $2^{k+1}$. Hence, $aC_{k+1}$ has order dividing $2^{k+1}$. We can then see that 
\begin{align*}
2aC_{k+1} + 2(2b - 1)D_{k+1} &= 0\\
aC_k + (2b - 1) D_k &= 0 \\
\implies a \equiv (2b - 1) &\equiv 0 \pmod{2^k},
\end{align*}
which is a contradiction. This implies that $\phi(E[2^{k+1}])$ is an index $2$ subgroup of $\langle C_{k+1}, D_{k+1} \rangle$ of order $2^{2k+1}$, and so $\langle C_{k+1}, D_{k+1} \rangle 
= E'[2^{k+1}]$. This proves the desired claim.
\end{proof}

Recall the maps $\rho_{E,2^k} : \Gal(K_k/\mathbb{Q}) \rightarrow
\AGL_2(\mathbb{Z}/2^k\mathbb{Z})$ and
$\tau : \Gal(K_k/\mathbb{Q}) \rightarrow
\GL_2(\mathbb{Z}/2^k\mathbb{Z})$,
defined at the beginning of this section. In \cite{RouseDZB},
an algorithm is given to compute the image of the $2$-adic Galois 
representation $\tau$. Running this algorithm shows that the image of $\tau$ (up
to conjugacy) is 
the index 6 subgroup of $\GL_2(\Bbb{Z}/2^k \Bbb{Z})$ generated by $\bmatrix 1 & 1 \\ 0 & 1 \endbmatrix$,  $\bmatrix 7 & 0 \\ 2 & 1 \endbmatrix$, and $\bmatrix 5 & 0 \\ 2 & 1 \endbmatrix$. Moreover, the subgroup generated by the
aforementioned matrices is the unique conjugate that corresponds to
the basis chosen in Theorem~\ref{choosebasis}.

\begin{theorem} 
\label{sqrt2}
If $\rho_{E,2^{k}}(\sigma) = (\vec{v}, M)$ where $\vec{v} = (e,f)$,
then $e \equiv 0 \pmod{2}$ if and only if $M \equiv 1, 7 \pmod{8}$.
\end{theorem}
\begin{proof}
We will show that $e \equiv 0 \pmod{2}$ and $M \equiv 1, 7 \pmod{8}$
if and only if $\sigma(\sqrt{2}) = \sqrt{2}$.

Let $\beta_{1}$ be a point in $E(K_{1})$ so that
$2 \beta_{1} = (2,2)$. We pick a basis $\langle A_{1}, B_{1} \rangle$
according to Theorem~\ref{choosebasis}. We have
$\sigma(\beta_{1}) = \beta_{1} + e A_{1} + f B_{1}$, where
$e, f \in \Z/2^{k} \Z$.

Let $\phi : E \to E'$ be the usual isogeny and note that
$B_{1} \in \ker \phi$.  Thus,
$\phi(\sigma(\beta_{1})) = \phi(\beta_{1} + e A_{1} + f B_{1}) =
\phi(\beta_{1}) + e \phi(A_{1})$.
It follows that $e \equiv 0 \pmod{2}$ if and only if
$\sigma(\phi(\beta_{1})) = \phi(\sigma(\beta_{1})) =
\phi(\beta_{1})$.
A straightforward computation shows that the coordinates of
$\phi(\beta_{1})$ generate $\Q(\sqrt{2})$.  It follows that
$e \equiv 0 \pmod{2}$ if and only if $\sigma(\sqrt{2}) = \sqrt{2}$.

Finally, suppose that $\sigma$ is the Artin symbol associated to a prime ideal
$\mathfrak{p}$ above a rational prime $p$. By properties
of the Weil pairing (see \cite{Silverman}, Section III.8),
we have that $\zeta_{2^{k}} = e^{2 \pi i / 2^{k}} \in \Q(E[2^{k}])$, and that
$\sigma(\zeta_{2^{k}}) = \zeta_{2^{k}}^{\det(M)} = \zeta_{2^{k}}^{p}$. Since
$\sqrt{2} = \zeta_{8} + \zeta_{8}^{-1}$, it follows
easily that $\sigma(\sqrt{2}) = \sqrt{2} \iff p \equiv 1, 7 \pmod{8}$
and hence $\sigma(\sqrt{2}) = \sqrt{2}$ if and only if $\det(M)
\equiv 1, 7 \pmod{8}$.
\end{proof}

For $k \geq 3$, define $I_k$ to be the subgroup of $\AGL_{2}(\Z/2^{k} \Z)$
whose elements are ordered pairs $\{(\vec v, M)\}$ where $\vec{v} = \bmatrix e & f \endbmatrix$, the reduction of $M \bmod 8$ is in the group generated by $\bmatrix 1 & 1 \\ 0 & 1 \endbmatrix$,  $\bmatrix 7 & 0 \\ 2 & 1 \endbmatrix$, and $\bmatrix 5 & 0 \\ 2 & 1 \endbmatrix$, and $e \equiv 0 \pmod{2}$ if and only if $\det(M) \equiv 1$ or $7 \pmod{8}$. By Theorem~\ref{sqrt2} and the
discussion preceeding it, we know that the image of $\rho : \Gal(K_{k}/\Q)
\to \AGL_{2}(\Z/2^{k} \Z)$ is contained in $I_{k}$.

We now aim to show that the map $\rho_{E,2^k} : \Gal(K_k/\Q) \rightarrow I_k$ is surjective for $k \geq 3$. By \cite{Silverman} (page 105), if we have an elliptic curve $E : y^2 = x^3 + Ax + B$, the division polynomial $\psi_m \in \Z[A, B, x, y]$ is determined recursively by:
\begin{align*}
\psi_1 &= 1, \psi_2 = 2y, \psi_3 = 3x^4 + 6Ax^2 + 12Bx - A^2, \\
\psi_4 &= 4y(x^6 + 5Ax^4 + 20Bx^3 - 5A^2x^2 - 4ABx - 8B^2 - A^3), \\
\psi_{2m+1} &= \psi_{m+2}\psi_m^3 - \psi_{m-1}\psi_{m+1}^3,\quad
2y\psi_{2m} = \psi_m(\psi_{m+1}\psi_{m-1}^2 - \psi_{m-2}\psi_{m+1}^2 ).
\end{align*}
We then define $\phi_m$ and $\omega_m$ as follows:
\begin{align*}
\phi_m &= x\psi_m^2 - \psi_{m+1}\psi_{m-1} \\
4y\omega_m &= \psi_{m+2}\psi_{m-1}^2 - \psi_{m-2}\psi_{m+1}^2.
\end{align*}
If $\Delta = -16(4A^3 + 27B^2) \neq 0$, then $\phi_m(x)$ and $\psi_m(x)^2$ are relatively prime. This also implies that, for $P = (x_0, y_0) \in E$, 
\begin{align*}
[m]P = \left( \frac{\phi_m(P)}{\psi_m(P)^2}, \frac{\omega_m(P)}{\psi_m(P)^3} \right). 
\end{align*}

\begin{lemma}
\label{i3lem}
The map $\rho_{E, 8} : \Gal(K_3, \mathbb{Q}) \rightarrow I_3$ is surjective.
\end{lemma}

\begin{proof}
The curve $E$ is isomorphic to $E_2 : y^2 = x^3 - 3267x +45630$. The isomorphism that takes $E$ to $E_2$ takes $P = (2,2) \text{ on } E$ to $P_2 = (87, 648) \text{ on } E_2$.

We use division polynomials to construct a polynomial $f(x)$ whose
roots are the $x$-coordinates of points $\beta_3$ on $E_2$ so that
$8\beta_3 = P_2$. By the above formulas,
$8P_2 = \left(\frac{\phi_8(P_2)}{\psi_8(P_2)^2},
  \frac{\omega_8(P_2)}{\psi_8(P_2)^3} \right) $.
Since $P_2 = (87, 648)$,
$$
f(x) = \phi_8(P_2) - 87 \psi_8(P_2)^2 = 0
$$
will yield the equation with roots that satisfy our requirement. This is a degree $64$ polynomial. By using Magma to compute the Galois Group of $f(x)$, we find the order to be $8192$. A simple calculation shows that $I_{3}$
has order $8192$ and since $f(x)$ splits in $K_{3}/\Q$, we have 
that $\Gal(K_{3}/\Q) \cong I_{3}$.
\end{proof}

To prove the surjectivity of $\rho_{E,2^{k}}$, we will consider the Frattini
subgroup of $I_{k}$. This is the intersection of all maximal subgroups
of $I_{k}$. Since $I_{k}$ is a $2$-group, every maximal subgroup is normal
and has index $2$. It follows from this that if $g \in I_{k}$,
then $g^{2} \in \Phi(I_{k})$. 

%
%
%\begin{lemma}
% Every element $(\vec v, I)$ with $\vec v \equiv 0 \pmod{4}$ is in $\Phi(I_k)$,% where $I$ is the identity matrix and $\Phi(I_k)$ is the intersection of all ma%ximal subgroups (the Frattini subgroup).
%\end{lemma}
%
%\begin{proof}
% Let $\vec v_2 = (4x, 4y) \equiv 0 \pmod{4}$ and $\vec v_1 = (2x, 2y)$ so that %$\vec v_2 = 2 \vec v_1$. Since $e_1$ is even, $\vec v_1$ can be paired with $I$% to form an element of $I_k$, as $\text{det}(I) \equiv 1 \pmod{8}$. So, $(\vec %v_1, I) \in I_k$. We know that if $g \in I_k$, then  $g^2 \in \Phi(I_k)$, so si%nce $(\vec v_1, I) \in I_k$, $(\vec v_1, I)^2 \in \Phi(I_k)$. 
%$$
%(\vec v_1, I)^2 = (\vec v_1 + \vec v_1I, I\cdot I) = (2 \vec v_1, I) \in \Phi(I%_k).
%$$
%Since $\vec v_2 = 2 \vec v_1$, $(\vec v_2, I) \in \Phi(I_k)$. 
%\end{proof}

\begin{lemma}
\label{fratlem}
For $3 \leq k$, $\Phi(I_k)$ contains all pairs $(\vec{v}, M)$ such that $\vec{v} \equiv \vec{0} \pmod{4}$ and  $M \equiv I \pmod{8}$. 
\end{lemma}
\begin{proof}
We begin by observing that for $r = k$, $(0, I) \in \Phi(I_k)$. We prove the
result by backwards induction on $r$. 

\textit{Inductive Hypothesis:} $\Phi(I_k)$ contains all pairs $(0,M), M \equiv I \pmod{2^r}$. Let $ g = I + 2^{r-2}N$, and let $h = I + 2^{r-1}N$. If $r \geq 5$, then a straightforward calculation shows that $(0, g) \in I_k$. So, $ (0, g)^2 = (0, g^2) \in \Phi(I_k)$. Therefore, for $r > 3$, 
\[
g^2 = I + 2^{r-1}N + 2^{2r-4}N^{2} \equiv h \pmod{2^{2r-4}}.
\]
By the induction hypothesis, 
$(0, g^2h^{-1}) \in \Phi(I_k)$, and so $(0, h) \in \Phi(I_k)$. 

So, for $k \geq r \geq 4$, all pairs $(0,M), M \equiv I \pmod{2^r} \in \Phi(I_k)$. We will now construct $I_4$, compute $\Phi(I_4)$, and determine if $\Phi(I_4) \in \{(\vec v, M) : \vec v \equiv 0 \pmod{8}, M \equiv I \pmod{8}\}$. A
computation with Magma shows that
$$
I_4 = \left \langle \bmatrix 1 & 1 & 0 \\ 0 & 1 & 0 \\ 0 & 0 & 1 \endbmatrix , \bmatrix 7 & 0 & 0 \\ 2 & 1 & 0 \\ 0 & 0 & 1 \endbmatrix , \bmatrix 5 & 0 & 0 \\ 2 & 1 & 0 \\ 1 & 0 & 1 \endbmatrix \right \rangle. 
$$
We then construct $\Phi(I_4)$ and then $\phi : \Phi(I_4) \rightarrow \GL_3(\mathbb{Z}/8\mathbb{Z})$. We check that $\ker \phi$ has order $64$ and this
proves the desired claim about $\Phi(I_{4})$.

Now, observe that if $\vec{v}_{1} = (2x,2y)$, then $(\vec{v}_{1}, I) \in I_{k}$
and so $(2 \vec{v}_{1}, I) = (\vec{v}_{1}, I)^{2} \in \Phi(I_{k})$,
and so $\Phi(I_{k})$ contains all pairs $(\vec{v}, I)$ with $\vec{v} \equiv \vec{0} \pmod{4}$. Finally, for any matrix $M \equiv I \pmod{8}$, we have
\[
  (\vec{v}_{1}, I) * (0, M) = (\vec{v}_{1}, M) \in \Phi(I_{k})
\]
and this proves the desired claim.
\end{proof}

Finally, we prove the desired surjectivity.

\begin{theorem}
The map $\rho_{E, 2^k} : \Gal(K_k/ \mathbb{Q}) \rightarrow I_k$ is surjective for all $k \geq 3$.
\end{theorem}
\begin{proof}
Suppose to the contrary the map is not surjective. 
Lemma~\ref{fratlem} implies that if $\rho_{E,2^{k}}$ is not surjective,
the image lies in a maximal subgroup $M$ which contains the kernel
of the map from $I_{k} \to I_{3}$, and so the image of $\rho_{E,8}$ must lie
in a maximal subgroup of $I_{3}$. This contardicts Lemma~\ref{i3lem},
and shows the map is surjective.
\end{proof}

Now, we indicate the relationship between $\rho_{E,2^{k}}$ and
$\rho_{E',2^{k}}$. Let $\sigma \in \Gal(K_k/\Q)$. If $\beta_{k}$
is chosen so $2^{k} \beta_{k} = P$, then
\begin{align*}
\sigma(A_k) &= aA_k + bB_k,\\
\sigma(B_k) &= cA_k + dB_k,\\
\sigma(\beta_k) &= \beta_k + eA_k + fB_k.
\end{align*}

Applying $\phi$ to these equations, we have
\begin{align*}
\phi(\sigma(A_k)) &= aC_k + 2bD_k = \sigma(\phi(A_k)) = \sigma(C_k),\\
\phi(\sigma(B_k)) &= cC_k + 2dD_k = \sigma(\phi(B_k)) = \sigma(2D_k), \\
\phi(\sigma(\beta_k)) &= \phi(\beta_k) + eC_k + 2fD_k = \sigma(\phi(\beta_k)) = \sigma(\beta_k'),
\end{align*}
where $2^k\beta_k' = R$ on $E'$. Using the relations from Theorem~\ref{choosebasis}, we have that $2D_{k} = D_{k-1}$ and $2C_{k} = C_{k-1}$. This gives
\begin{align*}
\sigma(C_{k-1}) &= aC_{k-1} + 2bD_{k-1}, \\
\sigma(D_{k-1}) &= \frac{c}{2}C_{k-1} + dD_{k-1}.
\end{align*}
Thus, the vector-matrix pair associated with $\rho_{E',2^{k-1}}$ is $(v',M')$, where $\vec v' = \bmatrix e & 2f \endbmatrix$ and $M' = \bmatrix a & 2b \\ \frac{c}{2} & d \endbmatrix$. 

Let $(v,M)$ be a vector-matrix pair in $I_k$. Suppose that $o$ is the
smallest non-negative integer so that $2^o \vec v$ is in the image of
$(I-M)$. Thus there are integers $c_{1}$ and $c_{2}$
(not necessarily unique) so that $2^{o} \vec{v} = c_{1} \vec{x}_{1} +
c_{2} \vec{x}_{2}$, where $\vec{x}_{1}$ and $\vec{x}_{2}$
are the first and second rows of $I-M$. 

\begin{lemma}
Assume that $\det(M-I) \not\equiv 0 \pmod{2^{k}}$. 
If $c_{1} \vec{x}_{1} + c_{2} \vec{x}_{2} = d_{1} \vec{x}_{1} + d_{2} \vec{x}_{2}$,
then $c_{1} \equiv d_{1} \pmod{2}$.
\end{lemma}
\begin{proof}
The assumption on $\det(M-I)$ implies that $\ker (M-I)$ has order
dividing $2^{k-1}$. However, if $c_{1} \vec{x}_{1} + c_{2} \vec{x}_{2}
= d_{1} \vec{x}_{1} + d_{2} \vec{x}_{2}$, then $\begin{bmatrix}
c_{1} - d_{1} & c_{2} - d_{2} \end{bmatrix}$ is an element of $\ker (M-I)$.
If $c_{1} \not\equiv d_{1} \pmod{2}$, then this element has order $2^{k}$,
which is a contradiction.
\end{proof}

The above lemma makes it so we can speak of $c_{1} \bmod 2$ unambigously.
We now have the following result.
\begin{theorem}
Assume the notation above. Let $o'$ be the smallest positive integer
so that $2^{o'} v'$ is in the image of $I-M$. If 
$\det(M-I) \not\equiv 0 \pmod{2^{k-1}}$, then $o \ne o'$ if and only if 
$c_{1}$ is even.
\end{theorem}
\begin{proof}
Let $\vec{y}_{1}$ and $\vec{y}_{2}$ be the first two rows of $I-M'$.
A straightforward calculation shows that if $2^{o} \vec{v}
= c_{1} \vec{x}_{1} + c_{2} \vec{x}_{2}$, then
$2^{o} \vec{v}' = c_{1} \vec{y}_{1} + 2c_{2} \vec{y}_{2}$. 
If $c_{1}$ is even, then it follows that $2^{o-1} \vec{v}'
= (c_{1}/2) \vec{y}_{1} + c_{2} \vec{y}_{2}$
and so $o \ne o'$.

Conversely, if $o \ne o'$, then $o' = o-1$ and so $2^{o-1} \vec{v}' = 
d_{1} \vec{y}_{1} + d_{2} \vec{y}_{2}$. We have then that
\[
  2^{o} \vec{v} \equiv 2d_{1} \vec{x}_{1} + d_{2} \vec{x}_{2} \pmod{2^{k-1}}.
\]
So if $\vec{x} = \begin{bmatrix} 2d_{1} \\ d_{2} \end{bmatrix}$
we have $\vec{x} (I-M) \equiv 2^{o} \vec{v} \pmod{2^{k-1}}$.
If there is a vector $\vec{x}'$ with $\vec{x} \not\equiv \vec{x}' \pmod{2}$
so that $\vec{x}' (I-M) \equiv 2^{o} \vec{v} \pmod{2^{k-1}}$, then
$\vec{x} - \vec{x}'$ is in the kernel of $I-M \pmod{2^{k-1}}$. However,
the order of $\vec{x} - \vec{x}'$ is $2^{k-1}$ and this contradicts
the condition on the determinant. This proves the desired result.
\end{proof}

\section{Proof of Theorem~\ref{main}}

Theorem~\ref{somosorder} states that a prime $p$ divides
a term in the Somos-5 sequence if and only if the order
of $(2,2) \in E(\Q)$ is different than the order of $(1,4) \in E'(\Q)$. 
Recall that $o$, the power of two dividing the order of $P$, is the smallest positive integer such that $2^o\vec{v} \in \text{im}(I-M)$, and $o'$ is the
power of two dividing the order of $R$. 

Suppose that $\det(I-M) \not\equiv 0 \pmod{2^{k-1}}$. We have $2^o \vec{v} \in \text{im}(I-M)$ if and only if $c_1\vec{x}_1 + c_2 \vec{x}_2 = 2^o \vec{v}$, where $M = \bmatrix a & b \\ c & d \endbmatrix$, $\vec{x}_1 = \bmatrix 1-a & -b \endbmatrix$, and $\vec{x}_2 = \bmatrix -c & 1-d \endbmatrix$. We know that $o \neq o'$ if and only if $c_{1}$ is even.
For the remainder of the argument, we will consider elements
of $I_{k}$ as $3 \times 3$ matrices $\begin{bmatrix}
a & b & 0 \\
c & d & 0 \\
e & f & 1 \end{bmatrix}$. We let
$I-M = \begin{bmatrix} \alpha & \beta & 0 \\
\gamma & \delta & 0 \\
e & f & 0 \end{bmatrix}$ and define $A = \gamma f - \delta e$, $B = \alpha f - \beta e$, and $C = \alpha \delta - \beta \gamma$. We define
$M_{3}^{0}(\Z/2^{k} \Z)$ to be the set of $3 \times 3$ matrices with entries
in $\Z/2^{k} \Z$ whose third column is zero. We will
use $\ord_{2}(r)$ to denote the highest power of $2$ dividing $r$ for
$r \in \Z/2^{k} \Z$. If $r = 0 \in \Z/2^{k} \Z$, we will interpret
$\ord_{2}(r)$ to have an undefined value, but we will declare
the inequality $\ord_{2}(r) \geq k$ to be true.

Solving the equation
$c_{1} \vec{x}_{1} + c_{2} \vec{x}_{2} = 2^{o} \vec{v}$ using Cramer's
rule gives that $c_{1} C = -2^{o} A$ and $c_{2} C = 2^{o} B$. Assuming
that $c_{1}$ is even and $o > 0$ implies that $c_{2}$ must be odd, and
this implies that $\ord_{2}(B) < \ord_{2}(C)$. Moreover, since the
power of $2$ dividing $c_{1} C$ must be higher than that dividing
$c_{2} C$ it follows that $\ord_{2}(B) < \ord_{2}(A)$. Conversely, if
$\ord_{2}(B) < \ord_{2}(A)$ and $\ord_{2}(B) < \ord_{2}(C)$, then
$o > 0$ and $c_{1}$ is even. Therefore, our goal is counting of
elements of $I_{k}$ with $\ord_2(A) > \ord_2(B)$ and
$\ord_2(C) > \ord_2(B)$. For an $M_{0} \in M_{3}^{0}(\Z/2^{r} \Z)$, define
\begin{align*}
\eta(M_0,r,k) &= \#\left\{M \in M_{3}^{0}(\Z/2^{k} \Z) : M \equiv M_0~\text{mod}(2^{r}), \ord_2(A), \ord_2(C) > \ord_2(B)\right\},\\
\mu(M_0, r) &= \lim_{k \to \infty} \frac{\eta(M_0, r, k)}{|I_3| \cdot 64^{k-3}}.
\end{align*}
Roughly speaking, $\mu(M_{0},r)$ is the fraction of
matrices $M \equiv M_{0} \pmod{2^{r}}$ in $I_{k}$ with the property that 
$\rho_{E,2^{k}}(\sigma_{p}) = M$ implies that $p$ divides a term of the
Somos-5 sequence.

\begin{theorem}
\label{limiteval}
We have
\[
  \lim_{x \to \infty} \frac{\pi'(x)}{\pi(x)}
  = \sum_{M \in I_{3}} \mu(I-M,3).
\]
\end{theorem}

Before we start the proof, we need some lemmas. The first is straightforward,
and we omit its proof.
\begin{lemma}
\label{count}
If $a \in \Z/2^{k} \Z$, then the number of pairs $(x,y) \in (\Z/2^{k} \Z)^{2}$
with $xy \equiv a \pmod{2^{k}}$ is $(\ord_{2}(a) + 1) 2^{k-1}$, where if
$a \equiv 0 \pmod{2^{k}}$, we take $\ord_{2}(a) = k+1$.
\end{lemma}
%\begin{proof}
%Suppose that $a \not\equiv 0 \pmod{2^{k}}$. It is clear that $0 \leq \ord_{2}(x) \leq \ord_{2}(a)$. We will count the number of solutions with $\ord_{2}(x) = r$.
%The original congruence is equivalent to $\frac{x}{2^{r}} y \equiv \frac{a}{2^{r}} \pmod{2^{k-r}}$. Our assumption implies that
%$x/2^{r}$ is odd, and so it has a multiplicative inverse mod $2^{k-r}$.
%This implies that $y \equiv \frac{a}{2^{r}} \cdot \left(\frac{x}{2^{r}}\right)^{-1}$. Hence, $y$ is determined modulo $2^{k-r}$. These values of $y \in \Z/2^{k-r} \Z$ lift to $2^{r}$ choices for $y$ in $\Z/2^{k} \Z$. Hence,
%for each choice of $x$ with $\ord_{2}(x) = r$, there are $2^{r}$ choices for
%$y$. The number of $x$ with $\ord_{2}(x) = r$ is $2^{k-1-r}$. Thus,
%the total number of solutions to $xy \equiv a \pmod{2^{k}}$
%with $\ord_{2}(x) = r$ is $2^{r} \cdot 2^{k-1-r} = 2^{k-1}$. The desired
%result follows by summing over $r$.
%
%In the case that $a \equiv 0 \pmod{2^{r}}$, we either
%have $x \equiv 0 \pmod{2^{k}}$ and $y$ can be anything (giving
%$2^{k}$ solutions), or we have $x \not\equiv 0 \pmod{2^{k}}$,
%and then we must have $y \equiv 0 \pmod{2^{k-\ord_{2}(x)}}$.
%This gives a total of $2^{\ord_{2}(x)}$ solutions for $y$. Again,
%the number of $x$ in $\Z/2^{k} \Z$ with $\ord_{2}(x) = r$ is
%$2^{k-1-r}$ and we get $2^{k-1}$ total solutions with $\ord_{2}(x) = r$,
%for $0 \leq r \leq k-1$. The number of solutions in this case
%is $k \cdot 2^{k-1} + 2^{k} = (k+2) 2^{k-1}$, as desired.
%\end{proof}

\begin{lemma}
\label{countmat}
The number of matrices $M \in M_{2}(\Z/2^{k} \Z)$ with
$\det(M) \equiv 0 \pmod{2^{k}}$ is $3 \cdot 2^{3k-1} - 2^{2k-1}$.
\end{lemma}
\begin{proof}
We counting quadruples $(a,b,c,d)$ with $ad \equiv bc \pmod{2^{k}}$. By Lemma~\ref{count}, this number
is equal to
\[
  \sum_{\alpha \in \Z/2^{k} \Z} \left((\ord_{2}(\alpha) + 1) 2^{k-1}\right)^{2},
\]
which can easily be shown to equal $3 \cdot 2^{3k-1} - 2^{2k-1}$.
\end{proof}

\begin{proof}[Proof of Theorem~\ref{limiteval}]
For $k \geq 1$, let $G = \Gal(K_{k}/\Q)$ and $\sigma \in G$
have the property that $\sigma = \artin{K_{k}/\Q}{\mathfrak{p}}$ for
some prime ideal $\mathfrak{p} \subseteq O_{K_{k}}$ with
$\mathfrak{p} \cap \Z = (p)$. Assume that $p$ is unramified in
$K_{k}/\Q$ and $E/\F_{p}$ has good reduction at $p$. Let
$M$ be the $3 \times 3$ matrix corresponding to $\rho_{E,2^{k}}(\sigma)$,
and $A$, $B$ and $C$ be the corresponding minors of
$I-M$. Then one of three alternatives occurs:

(a) $B \not\equiv 0 \pmod{2^{k}}$, and a higher power of $2$
divides both $A$ and $C$.

In this situation (the good case), previous results ensure that the order of $P$ in $E(\F_{p})$ is twice the
order of $R$ in $E'(\F_{p})$, and hence $p$ divides some term
in the Somos-5 sequence.

(b) One of $A$ or $C$ is not congruent to $0$ mod $2^{k}$
and the power of $2$ dividing $B$ is equal or higher than for $A$ or $C$.

In this situation (the bad case), previous results ensure
that the order of $P$ in $E(\F_{p})$ is equal to the order of $R$ in $E'(\F_{p})$
and $p$ does not divide any term in the Somos-5 sequence.

(c) $A \equiv B \equiv C \equiv 0 \pmod{2^{k}}$.

In this situation (the inconclusive case), we do not have enough information
to determine if $p$ divides a term in the Somos-5 sequence or not.

Fix $\epsilon > 0$ and choose a $k$ large enough so that both of the following
conditions are satisfied:

(i) $\left|\sum_{M \in I_{3}} \frac{\eta(I-M,3,k)}{|I_{3}| 64^{k-3}} - 
\sum_{M \in I_{3}} \mu(I-M,3)\right| < \epsilon/3$, and

(ii) the fraction of elements $M$ in $I_{3}$ with
$C \equiv \det(I-M) \equiv 0 \pmod{2^{k-1}}$ is less than $\epsilon/3$. 
(This fraction tends to zero by Lemma~\ref{countmat}.)

Let $\mathcal{C} \subseteq I_{k}$ be the collection of ``good'' elements
of $I_{k}$ and let $\mathcal{C}'$ be the collection of ``good or inconclusive''
elements. 

By the statements above, we have that
\[
  \sum_{M \in I_{3}} \mu(I-M,3) - 2\epsilon/3 < \frac{|\mathcal{C}|}{|I_{k}|} 
\]
and
\[
  \frac{|\mathcal{C}'|}{|I_{k}|}
  < \sum_{M \in I_{3}} \mu(I-M,3) + \epsilon/3.
\]
By the Chebotarev density theorem, we have
\[
  \lim_{x \to \infty} \frac{\# \{ p \text{ prime } : p \leq x \text{ is unramified in } K_{k} \text{ and } \artin{K_{k}/\Q}{p} \subseteq \mathcal{C} \}}{\pi(x)} = 
\frac{|\mathcal{C}|}{|I_{k}|},
\]
and the same with $\mathcal{C}'$.

Let $r$ be the number of primes that either ramify in $K_{k}/\Q$
or for which $E/\Q$ has bad reduction. Then
there is a constant $N$ so that if $x > N$, then
\[
  \sum_{M \in I_{3}} \mu(I-M,3) - \epsilon + \frac{r}{\pi(x)} < \frac{\# \{ p \text{ prime } : p \leq x \text{ is unramified in } K_{k} \text{ and } \artin{K_{k}/\Q}{p} \subseteq \mathcal{C} \}}{\pi(x)},
\]
and
\[
\frac{\# \{ p \text{ prime } : p \leq x \text{ is unramified in } K_{k} \text{ and } \artin{K_{k}/\Q}{p} \subseteq \mathcal{C}' \}}{\pi(x)} <
  \sum_{M \in I_{3}} \mu(I-M,3) + \epsilon - \frac{r}{\pi(x)}.
\]
It follows from these inequalities that for $x > N$, then
\[
  -\epsilon < \frac{\pi'(x)}{\pi(x)} - \sum_{M \in I_{3}} \mu(I-M,3) < \epsilon.
\]
This proves that
\[
  \lim_{x \to \infty} \frac{\pi'(x)}{\pi(x)} = \sum_{M \in I_{3}} \mu(I-M,3).
\]
\end{proof}

Our goal is now to compute $\sum_{M \in I_3} \mu(I-M, 3)$. To do
this, we will develop rules to compute $\mu(M,r)$ for any matrix
$M \in M_{3}(\Z/2^{r} \Z)$ whose third column is zero. Observe
that $\mu(M_{0},r) \leq \frac{\# \{ M \in M_{3}^{0}(\Z/2^{r} \Z) : M \equiv M_{0} \pmod{2^{r}} \}}{|I_{3}| \cdot 64^{r-3}} = \frac{1}{2 \cdot 64^{r-1}}$.

Also, if all the entires in $M$ are even,
then $\mu(M,r) = \frac{1}{64} \mu(\frac{M}{2}, r-1)$. This allows us
to reduce to matrices where at least one entry is odd. If $M \in M_{3}^{0}(\Z/2\Z)$ is the zero matrix, we have
\[
  \mu(M,1) = \frac{1}{64} \mu(M/2, 0)
  = \frac{1}{64} \sum_{N \in M_{3}^{0}(\Z/2\Z)} \mu(N,1)
  = \frac{1}{64} \mu(M,1) + \sum_{\substack{N \in M_{3}^{0}(\Z/2\Z) \\ N \ne M}}
  \mu(N,1).
\]
It follows that $\mu(M,1) = \frac{1}{63} \sum_{\substack{N \in M_{3}^{0}(\Z/2\Z) \\ N \ne M}} \mu(N,1)$.

In order to determine $\mu(M_{0},r)$, it is necessary to consider a
matrix $M \in M_{3}(\Z/2^{k} \Z)$ and examine the behavior of
matrices $M' \in M_{3}(\Z/2^{k+1} \Z)$ with $M' \equiv M \pmod{2^{k}}$.
We refer to these as `lifts' of $M$. We define
$A$, $B$ and $C$ to be functions defined on a matrix $M = \begin{bmatrix}
\alpha & \beta & 0 \\
\gamma & \delta & 0 \\
e & f & 0 \end{bmatrix}$, given by $A = \gamma f - \delta e$,
$B = \alpha f - \beta e$ and $C = \alpha \delta - \beta \gamma$.

\begin{theorem}
\label{dencomp}
Let $M = \begin{bmatrix} \alpha & \beta & 0 \\ \gamma & \delta & 0 \\ e & f & 0 \end{bmatrix} \in M_{3}^{0}(\Z/2^{k} \Z)$ and suppose
$A \equiv B \equiv C \equiv 0 \pmod{2^{k}}$.
\begin{enumerate}
\item If $\gamma$ or $\delta$ is odd, then $\mu(M,k) = 0$.
\item If $\gamma$ and $\delta$ are both even, but one of
$\alpha$, $\beta$, $e$ or $f$ is odd, then $\mu(M,k) = \frac{1}{6 \cdot 64^{k-1}}$.
\end{enumerate}
\end{theorem}
\begin{proof}
Suppose that $M \in M_{3}^{0}(\Z/2^{k} \Z)$ is a matrix with $\gamma$
or $\delta$ odd and with $A \equiv C \equiv 0 \pmod{2^{k}}$.
In the case that $\gamma$ is odd, the congruences $A \equiv 0 \pmod{2^{k}}$
and $C \equiv 0 \pmod{2^{k}}$ imply that $f \equiv \frac{e \delta}{\gamma} \pmod{2^{k}}$ and $\beta \equiv \frac{\alpha \delta}{\gamma} \pmod{2^{k}}$.
We then find that $B \equiv \alpha f - \beta e
\equiv \alpha \left(\frac{e \delta}{\gamma}\right) - \left(\frac{\alpha \delta}{\gamma}\right) e \equiv 0 \pmod{2^{k}}$. It follows that none of the
lifts of $M$ have $\ord_{2}(B) < \min\{ \ord_{2}(A), \ord_{2}(C) \}$ and
so $\mu(M,k) = 0$. A similar argument applies in the case that $\delta$ is odd.

Suppose now that $\gamma$ and $\delta$ are both even. Consider $M'$ to be a 
lift of $M$ mod $2^{k+1}$. Then we 
have
\[
  M' = \begin{bmatrix}
  \alpha' & \beta' & 0\\
  \gamma' & \delta' & 0\\
  e' & f' & 0 \end{bmatrix} =
\begin{bmatrix} \alpha + \alpha_{1} 2^{k} &
  \beta + \beta_{1} 2^{k} & 0 \\
  \gamma + \gamma_{1} 2^{k} & \delta + \delta_{1} 2^{k} & 0 \\
  e + e_{1} 2^{k} & f + f_{1} 2^{k} & 0 \end{bmatrix},
\]
where $\alpha_{1}, \beta_{1}, \gamma_{1}, \delta_{1}, e_{1}, f_{1} \in \F_{2}$.
If $A'$, $B'$ and $C'$ are the values of $A$, $B$, and $C$ associated to $M'$,
then
\begin{align*}
  A' \equiv A + 2^{k} (\gamma_{1} f - \delta_{1} e)
  \pmod{2^{k+1}}\\
  B' \equiv B + 2^{k} (\alpha_{1} f + \alpha f_{1} - \beta_{1} e - \beta e_{1}) \pmod{2^{k+1}}\\
  C' \equiv C + 2^{k} (\alpha \delta_{1} - \beta \gamma_{1}) \pmod{2^{k+1}}. 
\end{align*}

Suppose that $e$ or $f$ is odd. Then the map
$\F_{2}^{6} \to \F_{2}^{2}$ given by $(\alpha_{1}, \beta_{1}, \gamma_{1}, \delta_{1},
e_{1}, f_{1}) \mapsto (\gamma_{1} f - \delta_{1} e, \alpha_{1} f + \alpha f_{1} - \beta_{1} e - \beta e_{1})$
is surjective. It follows that of the $64$ lifts of $M'$,
one quarter have $(A' \bmod 2^{k+1}, B' \bmod 2^{k+1})$
equal to each of $(2^{k},2^{k})$, $(0,2^{k})$, $(2^{k},0)$ and $(0,0)$.
Moreover, if $A' \equiv 0 \pmod{2^{k+1}}$, then we must have $C'
\equiv 0 \pmod{2^{k+1}}$. This is because if $e'$ is odd,
then $\delta' \equiv \frac{\gamma' f'}{e'} \pmod{2^{k+1}}$,
and $\beta' \equiv \frac{\alpha f - B'}{e'} \pmod{2^{k+1}}$. Plugging
these into $C' = \alpha' \delta' - \beta' \gamma'$ gives
$C' \equiv \frac{B' \gamma'}{e'} \pmod{2^{k+1}}$. Since $\gamma'$ is even,
it follows that $C' \equiv 0 \pmod{2^{k+1}}$. A similar argument shows
that $C' \equiv 0 \pmod{2^{k+1}}$ if $f'$ is odd. As a consequence,
of the $64$ lifts of $M$, $32$ have $\mu(M',k+1) = 0$,
$16$ have $\ord_{2}(B') < \ord_{2}(A')$ and $\ord_{2}(B') < \ord_{2}(A')$.
For these, we have $\mu(M',k+1) = \frac{1}{2 \cdot 64^{k}}$.
The remainder have $A' \equiv B' \equiv C' \equiv 0 \pmod{2^{k+1}}$.
It follows that 
\[
  \mu(M,k) = \frac{1}{2 \cdot 64^{k-1}} \cdot \frac{1}{4}
+ \sum_{\substack{M' \equiv M \pmod{2^{k+1}} \\ 
  A' \equiv B' \equiv C' \equiv 0 \pmod{2^{k+1}}}} \mu(M',k+1).
\]
Applying the above argument repeatedly gives
\[
  \mu(M,k)
  = \frac{1}{2 \cdot 64^{k-1}}
  \cdot \left(\frac{1}{4} + \frac{1}{16} + \cdots
  + \frac{1}{4^{\ell}}\right)
  + \sum_{\substack{M' \equiv M \pmod{2^{k+\ell}}\\
  A' \equiv B' \equiv C' \equiv 0 \pmod{2^{k+\ell}}}} \mu(M',k+\ell).
\]
Using the bound $0 \leq \mu(M',k+\ell) \leq \frac{1}{2 \cdot 64^{k+\ell-1}}$,
noting that the sum contains $16^{\ell}$ terms, and
taking the limit as $\ell \to \infty$ yields that
$\mu(M,k) = \frac{1}{2 \cdot 64^{k-1}} \sum_{r=1}^{\infty} \frac{1}{4^{r}}
= \frac{1}{6 \cdot 64^{k-1}}$.

The case when $\alpha$ or $\beta$ is odd is very similar. In that
case, one can show that the $64$ lifts $M'$ have
$(B' \bmod 2^{k+1}, C' \bmod 2^{k+1})$ divided equally between
$(2^{k},2^{k}), (0,2^{k}), (2^{k},0)$ and $(0,0)$, and that
$C' \equiv 0 \pmod{2^{k+1}}$ implies that $A' \equiv 0 \pmod{2^{k+1}}$.
Again, one quarter of the lifts $M'$ have $B' \equiv 2^{k} \pmod{2^{k+1}}$
and $A' \equiv C' \equiv 0 \pmod{2^{k+1}}$, and
$\mu(M,k) = \frac{1}{6 \cdot 64^{k-1}}$.  
\end{proof}

Let $M \in M_{3}^{0}(\Z/8\Z)$ be the zero matrix. We have that
$\mu(M,3) = \frac{1}{64^{2}} \mu(M,1) = 
  \frac{1}{63} \cdot \frac{1}{64^{2}} \sum_{N \in M_{3}^{0}(\Z/2\Z)} \mu(N,1)$.
Of the $63$ nonzero matrices in $M_{3}^{0}(\Z/2\Z)$ we find that
$6$ have $B$ odd and $A$ and $C$ even, while $36$ have $A$ or $C$ odd.
Of the remaining $21$, there are $12$ that have $\gamma$ or $\delta$ odd,
and the remaining $9$ have $\gamma$ and $\delta$ both even.
It follows that
\[
  \mu(M,3) = \frac{1}{63} \cdot \frac{1}{64^{2}}
  \cdot \frac{1}{2} \cdot \left[ 6 + 36 \cdot 0 + 12 \cdot 0 + 9 \cdot \frac{1}{3} \right]
  = \frac{1}{8192} \cdot \frac{1}{7} = \frac{1}{57344}.
\]
(Note that in the denominator of $\mu(N,1)$ we have $|I_{3}| 64^{-2} = 8192 \cdot (1/4096) = 2$.)

For each of the $8191$ non-identity elements $M$ of $I_{3}$,
we divide $I-M$ by the highest power of $2$ dividing all of the elements, say $2^{r}$.
In $3754$ cases, we have $\ord_{2}(B) < \ord_{2}(A)$ and $\ord_{2}(B) < \ord_{2}(C)$. For each of these, $\mu(I-M,3) = \frac{1}{8192}$.

In $4036$ cases, we have $\ord_{2}(B) \geq \ord_{2}(A)$ or
$\ord_{2}(B) \geq \ord_{2}(C)$ and not all of $A$, $B$ and $C$
are congruent to $0$ modulo $2^{3-r}$. For each of these, $\mu(I-M,3) = 0$.

In $365$ cases, we have $A \equiv B \equiv C \equiv 0 \pmod{2^{3-r}}$
and $\gamma$ and $\delta$ are both even. In each of these cases,
$\mu(I-M,3) = \frac{1}{3 \cdot 8192}$ by Theorem~\ref{dencomp}.

In the remaining $36$ cases, we have $A \equiv B \equiv 0 \pmod{2^{3-r}}$
and one of $\gamma$ or $\delta$ is odd. By Theorem~\ref{dencomp},
$\mu(I-M,3) = 0$.

It follows that
\[
  \sum_{M \in I_{3}} \mu(I-M,3) = 3754 \cdot \frac{1}{8192}
  + 365 \cdot \frac{1}{3 \cdot 8192} + \frac{1}{57344}
  = \frac{5087}{10752}.
\]
This concludes the proof of Theorem~\ref{main}.

\bibliographystyle{plain}
\bibliography{somos5ref}

\end{document}